\documentclass[a4paper,11pt]{article}
\usepackage[utf8]{inputenc} 
\usepackage[T1]{fontenc} 
\usepackage[a4paper]{geometry} 
\usepackage{amsmath, amssymb, mathtools, amsthm}
\usepackage{mathtools} 
\usepackage{mathrsfs}  
\usepackage{graphicx} 
\usepackage{xcolor}   
\usepackage[english]{babel} 
\usepackage{float}
\usepackage{bbm}
\usepackage[title]{appendix}
\newcommand{\eps}{\varepsilon}
\usepackage{caption}
\usepackage{subcaption}
\usepackage{comment}
\usepackage{geometry}
\usepackage{dsfont}

\usepackage{enumitem} 
\usepackage{hyperref} 

\usepackage[sort]{cite}

\newcommand \commentout[1] {}

\newcommand{\R}{\mathbb{R}}

\newcommand{\Td}{\mathbb{T}^{d}}

\newcommand {\Chi} {{\bf \raise 2pt \hbox{$\chi$}} }

\newcommand{\ueps}{u_{\varepsilon}}

\newcommand {\f}   {\frac}
\newcommand {\p}   {\partial}

\newcommand*{\dd}{\mathop{\kern0pt\mathrm{d}}\!{}}

\newcommand*{\DD}{\mathop{\kern0pt\mathrm{D}}\!{}}

\DeclareMathOperator*{\supp}{\operatorname{supp}}

\theoremstyle{plain}
\newtheorem{theorem}{Theorem}[section]

\newtheorem{definition}[theorem]{Definition}
\newtheorem{lemma}[theorem]{Lemma}
\newtheorem{proposition}[theorem]{Proposition}

\theoremstyle{remark}
\newtheorem{remark}[theorem]{Remark}

\DeclareMathOperator{\DIV}{div}

\newcommand{\beq}{\begin{equation}}
\newcommand{\eeq}{\end{equation}}
\newcommand{\bea} {\begin{array}{rl}}
\newcommand{\eea} {\end{array}}
\newcommand{\bepa}{\left\{ \begin{array}{l}}
\newcommand{\eepa} {\end{array}\right.}
\newcommand{\diff}{\mathop{}\!\mathrm{d}}

\numberwithin{equation}{section}

\date{}

\begin{document}
\title{Nonlocal-to-local convergence of the Cahn-Hilliard equation with degenerate mobility and the Flory-Huggins potential.}

\author{Charles Elbar\thanks{Sorbonne Universit\'{e}, CNRS, Universit\'{e} de Paris, Inria, Laboratoire Jacques-Louis Lions, F-75005 Paris, France } \thanks{email: charles.elbar@sorbonne-universite.fr}
\and Jakub Skrzeczkowski\thanks{Mathematical Institute, University of Oxford, Woodstock Road, Oxford, OX2 6GG, United Kingdom; Faculty of Mathematics, Informatics and Mechanics, University of Warsaw, Stefana Banacha 2, 02-097 Warsaw, Poland} 
\thanks{email: jakub.skrzeczkowski@maths.ox.ac.uk}
}
\maketitle
\begin{abstract}
The Cahn-Hilliard equation is a fundamental model for phase separation phenomena. Its rigorous derivation from the nonlocal aggregation equation, motivated by the desire to link interacting particle systems and continuous descriptions, has received much attention in recent years. In the recent article, we showed how to treat the case of degenerate mobility for the first time. Here, we discuss how to adapt the exploited tools to the case of the mobility $m(u)=u\,(1-u)$ as in the original works of Giacomin-Lebowitz and Elliott-Garcke. The main additional information is the boundedness of $u$, implied by the form of mobility, which allows handling the nonlinear terms. We also discuss the case of (mildly) singular kernels and a model of cell-cell adhesion with the same mobility. 
\end{abstract}

\noindent{\makebox[1in]\hrulefill}\newline

\section{Introduction}

The question of nonlocal-to-local convergence for the Cahn-Hilliard equation is of fundamental nature, linking the models derived from particle processes or interacting particle systems \cite{MR1453735} and from physical principles \cite{CAHN1961795, Cahn-Hilliard-1958}. This question was recently addressed by several authors \cite{MR4198717, MR4093616,
MR4408204, MR4248454} but only for the case of constant mobility. Our recent article \cite{MR4574535} was the first one to establish rigorously the convergence for the degenerate mobility $m(u) = u$ and smooth potential as a consequence of nonlocal compactness results developed by Bourgain-Brezis-Mironcescu \cite{bourgain2001another} and Ponce \cite{MR2041005}. In the current paper we conclude our analysis by showing that the reasoning from \cite{MR4574535} can be extended to the mobility $m(u) = u\,(1-u)$ and the logarithmic Flory-Huggins potential as in the original works of Giacomin-Lebowitz \cite{MR1453735} and Elliot-Garcke \cite{MR1377481}. We also consider singular nonlocal interaction kernels in the spirit of \cite{MR4198717}. In the last section, using similar techniques, we also discuss the nonlocal-to-local limit of a model of cell-cell adhesion with the same mobility introduced in \cite{MR3948738}.   \\

We consider the nonlocal Cahn-Hilliard equation with degenerate mobility 
\begin{equation}\label{eq:CH_eq_general_without_eps}
\begin{split}
\partial_t u = \DIV(m(u) \nabla \mu)\quad \text{in}\quad &(0,+\infty)\times \Td,\\
\mu = B[u] + F'(u)\quad \text{in}\quad &(0,+\infty)\times \Td, 
\end{split}
\end{equation}
equipped with an initial datum $u^{0}\ge 0$. Here, the mobility is $m(u)=u(1-u)$, $\Td$ is the $d$-dimensional flat torus, which is particularly useful when treating nonlocal operators, $B$ is the nonlocal operator $B=B_{\varepsilon}$ defined with

\begin{equation}\label{operatorB}
B_\eps[u](x) = (J_{\eps}\ast 1)u(x)-J_{\eps}\ast u(x)=\int_{\Td}J_{\eps}(x-y)(u(x)-u(y)) \diff y.
\end{equation}

The convolution kernel $J_{\eps}$ is of the form 

\begin{equation}\label{eq:conv_kernel}
J_{\eps}(x)=\f{\omega_{\eps}(x)}{\eps^{2-\alpha}|x|^{\alpha}}  
\end{equation}
for $0\le\alpha<d-1$, where $\{\omega_\eps\}_\eps$ is the family of standard mollifiers:
\begin{equation}\label{eq:omega}
 \omega_{\eps}(x)=\f{1}{\eps^{d}}\omega\left(\f{x}{\eps}\right)    
\end{equation}
 with $\omega$ being a nonnegative smooth radially symmetric kernel which is compactly supported in the unit ball of $\R^{d}$ satisfying
\begin{equation}
\int_{\Td}\frac{\omega(y)y_{i}y_{j}}{|y|^{\alpha}} \diff y =\f{W}{d}\delta_{i,j},
\label{as:omega}
\end{equation}
for some $W>0$.

\begin{remark}\label{rem:kernel}
The range $\alpha\in[0,d-1)$ ensures that the kernel is in $W^{1,1}_{loc}(\R^{d})$ and therefore we can apply the existing theory of the existence of weak solutions for the nonlocal equation. This type of kernel has for instance been considered in~\cite{MR4248454}. 
\end{remark}

The potential $F$ is the usual double-well logarithmic potential 

\begin{equation}\label{eq:potential_F}
F(s)=s\log(s) + (1-s)\log(1-s) - \theta\left(s-\f{1}{2}\right)^{2}+k,  
\end{equation}
where $\theta>1$ and $k$ is a constant such that $F(s)\ge 0$ for all $s\in(0,1)$. Note in particular the identities
\begin{equation}\label{eq:identities_doublewell}
    F'(s)=\log\left(\f{s}{1-s}\right)+\theta(1-2s),\quad  F''(s)=\f{1}{s(1-s)}-2\theta, \quad m(s)F''(s)=1-2\theta s(1-s). 
\end{equation}

We impose the constraint on the parameter $\theta$
\begin{equation}\label{eq:constraint_theta}
2\,\theta<\inf_{\eps<\eps_{0}}(J_{\eps}\ast 1)
\end{equation} for some $\eps_{0}$. This is always satisfied for $\eps_{0}$ small enough since
\begin{equation*}
J_{\eps}\ast 1= \f{1}{\eps^{2}}\int_{\Td}\f{\omega(y)}{|y|^\alpha}\diff y.
\end{equation*}
As we intend to send $\eps \to 0$, we can assume that this condition is always satisfied. The condition \eqref{eq:constraint_theta} ensures that the coefficient $(J_{\eps}\ast 1) + F''(s)$ (standing next to $\nabla u_{\eps}$ in \eqref{eq:CH_eq_general_without_eps}) is strictly positive and the existence as well as uniqueness of the weak solution follows. \\

Our target is to prove that as $\varepsilon \to 0$, the solutions of 
 \begin{align}
\partial_t u_{\eps} = \DIV(m(u_{\eps}) \nabla \mu_{\eps})\quad \text{in}\quad &(0,+\infty)\times \Td,\label{eq:CHE1}\\
\mu_{\eps} = B_{\eps}[u_{\eps}] + F'(u_{\eps})\quad \text{in}\quad &(0,+\infty)\times \Td \label{eq:CHE2},
\end{align}
tend to the weak solution of the degenerate Cahn-Hilliard equation
 \begin{align}
\partial_t u = \DIV(m(u) \nabla \mu)\quad \text{in}\quad &(0,+\infty)\times \Td,\label{eq:CH1}\\
\mu = -\f{W}{2}\Delta u + F'(u)\quad \text{in}\quad &(0,+\infty)\times \Td \label{eq:CH2}.
\end{align}

Our main result reads as follows.

\begin{theorem}[Convergence of nonlocal to local Cahn-Hilliard equation on the torus]\label{thm:final}
Let $u^{0}\ge 0$ be an initial datum with finite energy and entropy $\sup_{\eps}E_{\eps}(u^0), \sup_{\eps}\Phi_{\eps}(u^0) < \infty$ defined in ~\eqref{eq:intro_energy} and \eqref{eq:intro_entropy}. Let $\{u_{\varepsilon}\}_{\eps}$, with $\eps<\min(\eps_{A},\eps_{B})$ defined in Lemmas~\ref{lem:Poincare_with_average}-\ref{lem:poincare_nonlocal_H1_L2}, be a sequence of solutions of the nonlocal Cahn-Hilliard equation~\eqref{eq:CHE1}-\eqref{eq:CHE2} as in Proposition~\ref{prop:existence_weak_sol}. Then, along a subsequence not relabeled,
$$
u_{\varepsilon} \to u \mbox{ in } L^2(0,T; H^1(\Td))\cap{L^{p}((0,T)\times\Td)}, \, \forall p\in[1,\infty),
$$
where $u$ is a weak solution of the degenerate Cahn-Hilliard equation~\eqref{eq:CH1}-\eqref{eq:CH2} as defined in Definition~\ref{def:weak_sol_limit}.   
\end{theorem}

\begin{remark}
If $u_0 \in H^{1}(\Td)$, then $\sup_{\eps}E_{\eps}(u^0), \sup_{\eps}\Phi_{\eps}(u^0) < \infty$, see~\cite[Theorem 1]{bourgain2001another}. 
\end{remark}

The main contribution of Theorem \ref{thm:final} is the possibility to consider more general, nonlinear mobilities $m(u)$. In \cite{MR4574535}, where we considered the particular mobility $m(u) = u$, we exploited the formula
$$
\nabla m(u) \, \Delta u = \nabla u \, \Delta u = \DIV(\nabla u \otimes \nabla u ) - \frac{1}{2} \nabla |\nabla u|^2
$$
in a crucial way. This is no longer available when $m(u) = u\,(1-u)$. However, the more general term can be still handled because the singular Flory-Huggings potential ensures that the sequence $\{u_{\eps}\}$ is uniformly bounded in $L^{\infty}((0,T)\times\Td)$.\\

We also would like to point out that while the setting above is formulated for the standard choice $m(u) = u\, (1-u)$, the same is also true for more general case $m(u) = u^k\, (1-u)^l$ where $k, l \geq 1$. In particular, the proof of Theorem \ref{thm:final} and of the crucial Lemma \ref{lem:conv_S_collected}, is formulated in so general way that it uses only the fact that $m$ is a $C^1$ function.\\

The structure of the paper is as follows. In Section \ref{sect:literature} we review the existing literature on the problem while in Section \ref{sect:existence} we recall the basic properties of the system \eqref{eq:CHE1}-\eqref{eq:CHE2} (well-posedness, energy, and entropy estimates). Section \ref{sect:convergence} is dedicated to the proof of the main results. In Section \ref{sect:numerics} we present numerical simulations which show certain qualitative differences between the local and nonlocal models. The last section, Section \ref{sect:cell_cell}, is devoted to a model of cell-cell adhesion phenomena from \cite{MR3948738}, which also includes the mobility $m(u) = u\,(1-u)$. For this model we also obtain a nonlocal-to-local convergence result.

\section{Motivations and litterature review}\label{sect:literature}
The non-local Cahn-Hilliard equation was initially derived by Giacomin and Lebowitz \cite{MR1453735,MR1638739} through a microscopic approach. Their model is based on a $d$-dimensional lattice gas that evolves through Kawasaki exchange dynamics, which is a Poisson process that exchanges nearest neighbors. In the hydrodynamic limit, the average of the occupation numbers over a small macroscopic volume element tends towards a solution of a non-local Cahn-Hilliard equation. This equation is an approximation of the local Cahn-Hilliard equation, as demonstrated in Theorem~\ref{thm:final}. The literature concerning the nonlocal Cahn-Hilliard equation is quite well-developed and we refer to \cite{MR3688414,MR3072989,MR1612250,KNOPF2021236,MR4365199,MR4221297,MR4562643,2023arXiv230307745P,MR3019479,MR3090070} and \cite{MR4241616,MR3346161,MR2746427,MR2009615,MR2784354,MR4227048,MR3518604,MR3903266} and the references therein for the cases of non-degenerate and degenerate mobilities respectively. For the local case, we refer for instance to~\cite{MR4001523,MR1377481,MR4199231,MR4432003}.\\

The convergence of the nonlocal Cahn-Hilliard equation to its local counterpart is a relatively new topic of research. The first result in this area was achieved in 2019 by Melchiona, Ranetbauer, Scarpa, and Trussardi on the torus with constant mobilities~\cite{MR4408204}. The study was subsequently extended to different types of kernels, potentials, and Neumann boundary conditions in~\cite{MR4093616,MR4198717,MR4248454,MR3362777}. Furthermore, due to the higher regularity of solutions for the constant mobility, very recently Abels and Hurm obtained an explicit rate of convergence in this case \cite{abels2023strong}. The proof does not work in the case of degenerate mobility.\\ 

Concerning the degenerate mobility, in~\cite{MR4574535}, we investigated the case $m(u)=u$ with a smooth potential, motivated by biological applications \cite{CRMECA, MR4745662}. This result has been extended to the case of systems in~\cite{2023arXiv230311929C} which appear in modeling of cell-cell adhesion \cite{falco2022local,MR3948738}. The result presented in this paper (Theorem \ref{thm:final}) completes the nonlocal-to-local convergence program by examining the mobility $m(u)=u(1-u)$ and logarithmic potential, which is a typical framework in the analysis of the Cahn-Hilliard equation.\\

Finally, we wish to put our research in a broader context. First, we remark that \eqref{eq:CHE1}--\eqref{eq:CHE2} can be interpreted as a gradient flow in the Wasserstein space with nonlinear mobility. Such problems have been extensively studied recently \cite{MR2565840, MR3659834, MR3558359, MR2921215, MR2745794, MR2672546,
MR2448650,carrillo2023structure}. Second, the question of passing to the limit from a nonlocal equation to a local one attracts researchers studying different equations. This includes porous media equation \cite{carrillo2023nonlocal, burger2022porous, MR1821479,Hecht2023porous,MR3913840}, hyperbolic conservation laws \cite{coclite2023ole} and cross-diffusion systems \cite{elbar2023inviscid, david2023degenerate}.

\section{Existence of solutions}\label{sect:existence}
In~\cite{MR1638739}, Giacomin and Lebowitz proved the existence and uniqueness of a weak solution to~\eqref{eq:CHE1}-\eqref{eq:CHE2} on the torus. In~\cite{MR4241616} Frigeri, Gal, and Grasselli obtain the existence and uniqueness in the case of a bounded domain with Neumann boundary conditions. The following two propositions can be adapted from their work. 

\begin{proposition}[Existence of weak solutions]
\label{prop:existence_weak_sol} {~\cite[Theorem 4.1]{MR1638739} or \cite[Theorem 2.3]{MR4241616}}. 
 Let $u^{0}$ be a measurable function such that $F(u_0)\in L^{1}(\Td)$ and $\phi(u_{0})\in L^{1}(\Td)$ where $\phi\in C^{2}(0,1)$ is defined as $m''(s)\phi(s)=1$, $\phi(0)=\phi'(0)=0$. Then there exists a unique weak solution $u$ of~\eqref{eq:CHE1}-\eqref{eq:CHE2} such that 
$u \in L^{2}(0,T;H^{1}(\Td))$, $0\le u\le 1$ a.e. in $(0,T)\times\Td$ and for all $\varphi\in C_c^{\infty}([0,T)\times\Td)$ we have
\begin{multline*}
 -\int_0^T \int_{\Td} u\, \p_t \varphi \diff t - \int_{\Td} u^0(x)\, \varphi(0,x)\diff x+\int_0^T\int_{\Td}m(u)F''(u)\nabla u\nabla\varphi\diff x\diff t+ \\ +\int_0^T \int_{\Td}m(u)\nabla ((J_{\eps}\ast 1)u- J_{\eps}\ast u) \nabla\varphi\diff x \diff t=0.
\end{multline*} 
\end{proposition}

\begin{remark}
The weak solutions defined in Proposition~\ref{prop:existence_weak_sol} can be shown to have higher regularity (under some assumptions on the initial condition). In particular, it is possible to prove the existence/uniqueness of a strong solution satisfying the strict separation property~\cite{MR4241616}. 
\end{remark}
\phantom{...}\\
Let us sketch quickly the proof of the bound $0\leq u \leq 1$ which is the most important for the current work. Let $0\leq u^0 \leq 1$ and let $\mbox{sign}_+(u-1) = \mathds{1}_{u>1}$. Multiplying \eqref{eq:CHE1}-\eqref{eq:CHE2} by $\mbox{sign}_+(u-1)$ and using the notation $|u-1|_+ = (u-1)\mathds{1}_{u>1}$ we get
$$
\p_t |u_{\eps}-1|_+ = -\DIV(u_{\eps}\,|u_{\eps}-1|_+\,\nabla \mu_{\eps})
$$
so that $\partial_t \int_{\Td} |u_{\eps}-1|_+ \diff x = 0$. The bound follows from the one on the initial condition.\\

It can be further proved that the solutions in Proposition \ref{prop:existence_weak_sol} satisfy the following energy/entropy estimates (at least as inequalities) uniformly in $\eps$:

\begin{proposition}[Energy/ entropy estimates]\label{prop:en/ent}
Let $u_{\eps}$ be a sequence of weak solutions as defined in Proposition~\ref{prop:existence_weak_sol} with same initial condition $u^{0}$ (one can also consider a sequence of initial conditions which converges strongly). Then, defining the energy and entropy:
\begin{align}
&E_{\varepsilon}[u]:=\int_{\Td}F(u)\diff x+\f{1}{4}\int_{\Td}\int_{\Td}\f{\omega_{\eps}(y)}{\eps^{2-\alpha}|y|^\alpha}|u(x)-u(x-y)|^{2}\diff x\diff y, \label{eq:intro_energy}
\\ &\Phi[u]:=\int_{\Td}\phi(u)\diff x,\quad \phi''(u)=\f{1}{m(u)}, \,\phi(0)=\phi'(0)=0, \label{eq:intro_entropy}
\end{align}
we have
\begin{align}
&\int_{\Td}u_{\eps}(t,\cdot)\diff x=\int_{\Td}u^{0}\diff x,\label{PWreg2} \\
&\f{d}{dt} E[u_\eps] + \int_{\Td} m(u_{\eps}) \, |\nabla \mu_\eps|^2 \le 0,\label{eq:energy}\\
&\f{d}{dt}\Phi[u_\eps] + \frac{1}{2} \int_{\Td} \int_{\Td} \f{\omega_{\varepsilon}(y)}{\eps^{2-\alpha}|y|^\alpha} \, |\nabla u_{\varepsilon}(x) - \nabla u_{\varepsilon}(x-y)|^2 \diff x \diff y + \int_{\Td} F''(u_{\eps}) |\nabla u_{\eps}|^{2} \diff x\le 0,\label{eq:entropy}
\end{align}
where $\mu_{\eps}$ is the chemical potential defined in~\eqref{eq:CHE2}.  
\end{proposition}

Roughly speaking, to see these estimates, one multiplies Equation~\eqref{eq:CHE1} by $\mu$ to obtain the energy and by $\phi'(u)$ to obtain the entropy.\\

Since $F$ is not convex (i.e. $F''$ can be negative), we need to justify that the entropy estimates~\eqref{eq:entropy} indeed provide a priori estimates. More precisely, we need to prove that the negative part of the integral $\int_{\Td} F''(u_{\eps}) |\nabla u_{\eps}|^{2} \diff x$ is controlled where the negative part of $F''$ equals $-2\,\theta$. By \eqref{eq:energy}, we know that the sequence $\{u_{\eps}\}$ is bounded in $L^{\infty}(0,T; L^2(\Td))$. Hence, the sufficient control comes from Lemma \ref{lem:poincare_nonlocal_H1_L2} with $\gamma = \frac{1}{8\theta}$ 
\begin{equation}\label{eq:control_negative_part_entropy}
2\,\theta \int_{\Td}|\nabla u_{\eps} |^{2} \leq \frac{1}{4}\int_{\Td} \int_{\Td}  \frac{|\nabla u_{\eps}(x) - \nabla u_{\eps}(x-y)|^2}{\eps^{2-\alpha}|y|^\alpha} \omega_\eps(|y|) \diff x \diff y + C(\theta) \|u_{\eps} \|^2_{L^2(\Td)}.
\end{equation}

With the existence of a unique weak solution with energy/entropy bounds for all $\eps$ small enough, we are now in a position to pass to the local limit $\eps\to 0$.

\section{Convergence $\eps\to 0$}\label{sect:convergence}

We first define the weak solutions to the local Cahn-Hilliard equation. 

\begin{definition}\label{def:weak_sol_limit}
We say that $u$ is a weak solution of~\eqref{eq:CH1}-\eqref{eq:CH2} if 
$$
u \in L^{2}(0,T;H^{2}(\Td))\cap L^{\infty}(0,T;H^{1}(\Td)),
$$
$$
|u|\le 1\, \text{a.e. in $(0,T)\times\Td$},
$$
and if for all $\varphi \in C_c^{\infty}([0,T) \times \Td)$ we have 
\begin{multline*}
-\int_0^T \int_{\Td} u\, \p_t \varphi \diff x \diff t -\int_{\Td} u^0(x)\, \varphi(0,x) \diff x   = -\f{W}{2}\int_0^T \int_{\Td} \Delta u\, \nabla m(u) \cdot \nabla\varphi \diff x \diff t \\
-\f{W}{2}\int_0^T \int_{\Td}m(u)\,\Delta u\,\Delta\varphi \diff x \diff t-\int_0^T \int_{\Td}m(u)\,F''(u)\,\nabla u\cdot\nabla\varphi \diff x \diff t.    
\end{multline*}
\end{definition}

\subsection{Uniform estimates and compactness}
In order to prove Theorem~\ref{thm:final} we first collect uniform estimates in $\eps$: 

\begin{lemma}[Uniform estimates]\label{lem:uniform_est_just_eps}
The following sequences are bounded uniformly in $\eps$:
\begin{enumerate}[label=(\Alph*)]
    \item\label{item_est1} $|u_\eps|\le 1$\, a.e. in $(0,T)\times\Td$,
    \item\label{item_est2} $\{\nabla u_\eps\}_{\eps}$ in $L^2((0,T)\times \Td)$,
    \item\label{item_est3} $\{\sqrt{m(u_\eps)}\, \nabla \mu_\eps\}_{\eps}$ in $L^2((0,T)\times \Td)$,
    \item\label{item_est4} $\{\partial_t u_\eps\}_{\eps}$ in $L^2(0,T; H^{-1}(\Td))$,
    \item\label{item_est5} $\{\partial_t \nabla u_\eps\}_{\eps}$ in $L^2(0,T; H^{-2}(\Td))$,
    \item\label{item_est6}$\left\{\frac{\sqrt{\omega(y)}}{|y|^{\frac{\alpha}{2}}}\, \frac{ u_{\eps}(x) -  u_{\eps}(x-\eps\,y)}{\eps}\right\}_{\eps}$ in $L^{\infty}(0,T; L^2(\Td\times\Td))$,
    \item\label{item_est7}
    $\left\{\frac{\sqrt{\omega(y)}}{|y|^{\frac{\alpha}{2}}}\, \frac{\nabla u_{\eps}(x) - \nabla u_{\eps}(x-\eps\,y)}{\eps}\right\}_{\eps}$ in $L^{2}((0,T)\times\Td\times\Td)$.
    
\end{enumerate} 
\end{lemma}

\begin{proof}
The estimates are a consequence of Propositions~\ref{prop:existence_weak_sol}, \ref{prop:en/ent} and Lemma~\ref{lem:Poincare_with_average}. To be more precise, \ref{item_est1} is a consequence of Proposition \ref{prop:existence_weak_sol} while \ref{item_est3}, \ref{item_est6}, \ref{item_est7} follows from the energy/entropy identities \eqref{eq:energy}--\eqref{eq:entropy} and a change of variables $y\to \eps y$ (note that the negative part in the entropy identity \eqref{eq:entropy} is controlled as explained in \eqref{eq:entropy}). Then, \ref{item_est2} can be easily deduced by the Poincaré's inequality in Lemma \ref{lem:Poincare_with_average} because the average of the gradient on the torus is 0. Next, \ref{item_est4} follow from the PDE \eqref{eq:CHE1}--\eqref{eq:CHE2} by writing
$$
\partial_t u_{\eps} = \DIV(\sqrt{m(\ueps)}\,\sqrt{m(\ueps)}\,\nabla \ueps)
$$
and using estimates \ref{item_est1}, \ref{item_est3}. Finally, \ref{item_est5} is a direct consequence of \ref{item_est4}.\\
\end{proof}

\begin{lemma}[compactness]\label{lem:compact_ueps_collected}
    There exists $u\in L^{\infty}((0,T)\times\Td)\cap L^{2}(0,T;H^{2}(\Td))$, $0\le u\le 1$, such that, up to a subsequence not relabeled, $
    u_{\eps}\to u \mbox{ a.e., strongly in } 
    L^{2}(0,T;H^{1}(\Td))$ and $L^{p}((0,T)\times\Td) \, \mbox{ for all } p \in [1,\infty).$
\end{lemma}
\begin{proof}
First, by a classical Aubin-Lions lemma and estimates~\ref{item_est2}, \ref{item_est4} we obtain strong convergence in $L^{2}((0,T)\times\Td)$ (and a.e.). Then, arguing by interpolation and \ref{item_est1}, we deduce convergence in all $L^{p}((0,T)\times\Td)$. For the strong convergence in $L^{2}(0,T;H^{1}(\Td))$, we first obtain compactness in space of the gradient with~\ref{item_est7} and Theorem~\ref{thm:ponce_tx}. Compactness in time follows from~\ref{item_est5}. An Aubin-Lions argument yields the result (see \cite[Lemma D.1]{MR4574535} for the precise statement). 
\end{proof}

\subsection{Nonlocal gradient $S_{\eps}$ and its convergence properties}
The limit equation that we obtain is a fourth-order equation. The number of derivatives can be reduced by the integration by parts in a weak formulation. Therefore, to obtain this weak formulation we need to mimic the same integration by parts at the nonlocal level (i.e. with $B_{\eps}[u]$ instead of $-\Delta u$). This can be done if we introduce a nonlocal approximation of $\nabla$. Thus we define the operator (see e.g.~\cite{MR4574535})
\begin{equation}\label{eq:nonlocal_gradient}
S_{\eps}[\varphi](x,y):=\f{\sqrt{\omega_{\eps}(y)}}{\sqrt{2}\, \eps^{1-\frac{\alpha}{2}}|y|^{\frac{\alpha}{2}}}(\varphi(x-y)-\varphi(x)), 
\end{equation}
which has the following algebraic properties:
\begin{lemma}\label{lem:S_properties} The operator $S_{\eps}$ satisfies:
\begin{enumerate}[label=(S\arabic*)]
    \item $S_{\eps}$ is a linear operator that commutes with derivatives with respect to $x$,
    \item\label{propS_product_rule} for all functions $f,g: \Td \to \R$ we have the Leibniz rule up to an error:
    \begin{align*}S_{\eps}[fg](x,y)-S_{\eps}[f](x,y)g(x)&-S_{\eps}[g](x,y)f(x)= \\ &= \f{\sqrt{\omega_{\eps}(y)}}{\sqrt{2}\, \eps^{1-\frac{\alpha}{2}}|y|^{\frac{\alpha}{2}}}(f(x-y)-f(x)) \, (g(x-y)-g(x)).
    \end{align*}
    \item\label{propS_nonneg} for all $u, \varphi \in L^2(\Td)$
    $$
    \langle B_{\eps}[u](\cdot),\varphi(\cdot)\rangle_{L^{2}(\Td)}=\langle S_{\eps}[u](\cdot,\cdot),S_{\eps}[\varphi](\cdot,\cdot)\rangle_{L^{2}(\Td \times \Td)}.
    $$
\end{enumerate}
\end{lemma}
For the simple proof, we refer to \cite[Lemma 3.4]{MR4574535}.\\

Due to the property \ref{propS_nonneg}, when proving the convergence $\eps \to 0$, we will need to deal with the operator $S_{\eps}$. Its asymptotic properties are summarized in the following result. We remark that it is a substantial upgrade of \cite[Lemma 3.4 (S4)]{MR4574535} in two directions: allowing to consider singular potentials and including the case of general mobility $m(u)$. To simplify the notation, we skip $\diff y \diff x \diff t$ in the integrals in the statement of Lemma \ref{lem:conv_S_collected} and its proof.

\begin{lemma}\label{lem:conv_S_collected}
Let $f\in C^{1}(\R, \R)$. Suppose that $\{u_{\varepsilon}\}$ converges strongly in $L^2(0,T; H^1(\Td))$ to some $u \in L^{2}(0,T;H^{2}(\Td))$. Moreover, assume that $0\le u_{\eps} \le 1$ and that the sequence
\begin{equation}\label{eq:boundedness_gradients}
\left\{\frac{\sqrt{\omega(y)}}{|y|^{\frac{\alpha}{2}}}\, \frac{\nabla u_{\eps}(x) - \nabla u_{\eps}(x-\eps\,y)}{\eps}\right\} \mbox{ is bounded in } L^2((0,T)\times\Td\times\Td).
\end{equation}
Let $W := \int_{\Td} \frac{\omega(y)}{|y|^{\alpha-2}} \diff y$. Then, when $\eps\to 0$, for all smooth $\varphi(t,x): [0,T]\times\Td \to \R$:
    \begin{align}
     &\int_0^T \int_{\Td} \int_{\Td} S_{\varepsilon}[f(u_{\eps})]\, S_{\eps}[\nabla u_{\eps}] \cdot \nabla \varphi(t,x) \to \f{W}{2}\, \int_0^T \int_{\Td} \nabla f(u(t,x)) \cdot D^{2} u(t,x) \nabla \varphi(t,x) \label{Sconv1},\\
     &\int_0^T \int_{\Td} \int_{\Td} S_{\eps}[f(u_{\eps})] \,  S_{\varepsilon}[u_{\varepsilon}] \, \varphi(t,x) \to \f{W}{2}\, \int_0^T \int_{\Td} \nabla f(u(t,x))\cdot \nabla u(t,x) \, \varphi(t,x)\label{Sconv5},\\
    &\int_0^T \int_{\Td} \int_{\Td} S_{\varepsilon}[u_{\varepsilon}]  \nabla f(u_{\eps}) \cdot S_{\eps}[\nabla \varphi(t,x)] \to \f{W}{2}\, \int_0^T \int_{\Td} \nabla f(u(t,x))\cdot  D^2\varphi(t,x) \nabla u(t,x)\label{Sconv6},\\
    &\int_0^T \int_{\Td} \int_{\Td} S_{\varepsilon}[u_{\varepsilon}]  f(u_{\eps}) \, S_{\eps}[\varphi(t,x)] \to \f{W}{2}\, \int_0^T \int_{\Td}f(u(t,x))\,  \nabla u(t,x)\cdot\nabla\varphi(t,x).\label{Sconv4}
    \end{align}
\end{lemma}

\begin{proof}
We write
\begin{multline*}
\int_0^T \int_{\Td} \int_{\Td} S_{\varepsilon}[f(u_{\eps})]S_{\eps}[\nabla u_{\eps}] \cdot \nabla \varphi(x) =\\
=\f{1}{2}\int_0^T \int_{\Td} \int_{\Td}\f{\omega(y)}{|y|^{\alpha}}\f{ f(u_{\eps}(x))- f(u_{\eps}(x-\eps y))}{\eps}\f{\nabla u_{\eps}(x)-\nabla u_{\eps}(x-\eps y)}{\eps} \cdot \nabla \varphi(x).
\end{multline*}
From Lemma~\ref{lem:diff_quot_strong_conv}  we know that 
\begin{equation}\label{eq:compactness_diff_quotients}
\f{v_{\eps}(x)-v_{\eps}(x-\eps y)}{\eps\,|y|} \to 
\nabla v \cdot \frac{y}{|y|} \mbox{ strongly in }  L^{\infty}_{y}(\Td; L^2_{(t,x)}((0,T)\times\Td))
\end{equation}
for any sequence $\{v_{\eps}\}$ compact in $L^2(0,T;H^1(\Td))$ and converging to some $v$, in particular for $v_{\eps} = f(u_{\eps})$. Moreover, by \eqref{eq:boundedness_gradients}, we know that $\f{\sqrt{\omega(y)}}{|y|^{\frac{\alpha}{2}}}\f{\nabla u_{\eps}(x)-\nabla u_{\eps}(x-\eps y)}{\eps}$ is uniformly bounded in $L^{2}((0,T)\times\Td\times\Td)$ so that, up to a subsequence, it has a weak limit. To identify the limit, we compute for an arbitrary smooth vector field $\phi = \phi(t,x,y)$
\begin{multline*}
\int_{0}^{T}\int_{\Td}\int_{\Td}  \f{\sqrt{\omega(y)}}{|y|^{\frac{\alpha}{2}}}\f{\nabla u_{\eps}(x)-\nabla u_{\eps}(x-\eps y)}{\eps} \cdot \phi  = -\int_{0}^{T}\int_{\Td}\int_{\Td}  \f{\sqrt{\omega(y)}}{|y|^{\frac{\alpha}{2}}}\f{u_{\eps}(x)-u_{\eps}(x-\eps y)}{\eps} \DIV_x \phi.
\end{multline*}
Using \eqref{eq:compactness_diff_quotients} (with $v_{\eps} = u_{\eps}$) and $\frac{\sqrt{\omega(y)}}{|y|^{\frac{\alpha}{2} -1 }} \in L^2(\Td) \subset L^1(\Td)$ (because we assume $\alpha<d-1$ but even $\alpha<d+2$ would be sufficient here) we can pass to the limit and obtain
$$
\lim_{\eps\to 0} \int_{0}^{T}\int_{\Td}\int_{\Td}  \f{\sqrt{\omega(y)}}{|y|^{\frac{\alpha}{2}}}\f{u_{\eps}(x)-u_{\eps}(x-\eps y)}{\eps} \DIV_x \phi =  \int_{0}^{T}\int_{\Td}\int_{\Td} \f{\sqrt{\omega(y)}}{|y|^{\frac{\alpha}{2}}} \nabla u(x) \cdot y \DIV_x \phi
$$
Moreover, using $H^2$ regularity of $u$, we can integrate by parts backwards to conclude
\begin{equation}\label{eq:weak_conv_diff_quo_grad}
\f{\sqrt{\omega(y)}}{|y|^{\frac{\alpha}{2}}}\f{\nabla u_{\eps}(x)-\nabla u_{\eps}(x-\eps y)}{\eps} \rightharpoonup \f{\sqrt{\omega(y)}}{|y|^{\frac{\alpha}{2}}} \, D^2 u(x) y   \mbox{ in }  L^{2}((0,T)\times\Td\times\Td).
\end{equation}
Now, it remains to use again that $\frac{\sqrt{\omega(y)}}{{|y|^{\frac{\alpha}{2}-1}}} \in L^2(\Td)$ and apply \eqref{eq:compactness_diff_quotients}--\eqref{eq:weak_conv_diff_quo_grad}.\\

Concerning~\eqref{Sconv5}, we compute
\begin{multline*}
\int_0^T \int_{\Td} \int_{\Td} S_{\varepsilon}[u_{\varepsilon}] S_{\eps}[f(u_{\eps})] \, \varphi(t,x)= \\ =\int_0^T \int_{\Td} \int_{\Td}\f{\omega(y)}{2\,|y|^{\alpha-2}}\f{u_{\eps}(x)-u_{\eps}(x-\eps y)}{\eps|y|}\f{f(u_{\eps}(x))-f(u_{\eps}(x-\eps y))}{\eps|y|} \, \varphi(t,x)
\end{multline*}
Clearly $\frac{\omega(y)}{|y|^{\alpha-2}} \in L^1(\Td)$. Since $f$ is $C^{1}$ and $0\le u_{\eps}\le 1$, we can apply \eqref{eq:compactness_diff_quotients} both to $u_{\eps}$ and to $f(u_{\eps})$ so that we obtain
\begin{align*}
\int_0^T \int_{\Td} \int_{\Td} S_{\varepsilon}[u_{\varepsilon}] S_{\eps}[f(u_{\eps})] \, \varphi(t,x)\to \frac{W}{2}\, \int_0^T \int_{\Td} \nabla u(x)\cdot \nabla f(u(x)) \, \varphi(t,x).
\end{align*}
Next, we demonstrate~\eqref{Sconv6}. We write the integral as 
\begin{equation*}
\int_0^T \int_{\Td} \int_{\Td} \f{\omega(y)}{2|y|^{\alpha-2}}\f{u_{\eps}(x)-u_{\eps}(x-\eps y)}{\eps\,|y|} \, \nabla f(u_{\eps}(x)) \cdot \f{\nabla \varphi(t,x)-\nabla \varphi(t,x-\eps y)}{\eps\,|y|}.
\end{equation*}
Observe that $\nabla f(u_{\eps})$ is weakly compact in $L^2((0,T)\times\Td)$, $\f{u_{\eps}(x)-u_{\eps}(x-\eps y)}{\eps\,|y|}$ is compact due to \eqref{eq:compactness_diff_quotients} while $\f{\varphi(t,x)-\varphi(t,x-\eps y)}{\eps\,|y|}$ is compact in $L^{\infty}((0,T)\times\Td\times\Td)$. Hence, using integrability of $\f{\omega(y)}{2|y|^{\alpha-2}}$ as above, we easily pass to the limit $\eps \to 0$. \\

Finally, the proof of \eqref{Sconv4} is similar to that of \eqref{Sconv6}: indeed, we have the same compactness estimates for the sequence $\{f(u_{\eps})\}$ as for $\{\nabla f(u_{\eps})\}$. 
\end{proof}

\begin{proof}[Proof of Theorem \ref{thm:final}] We only have to explain how to pass to the limit in the term $$
\int_{0}^{T}\int_{\Td}\DIV(m(u_{\eps})\nabla\mu_{\eps})\,\varphi\diff x\diff t
$$ 
where $\varphi \in C^3([0,T]\times \Td)$. Integrating by parts, we obtain 
\begin{equation}\label{eq:split_for_I1_I2_I3}
    \begin{split}
\int_{0}^{T}\int_{\Td}\DIV&(m(u_{\eps})\nabla\mu_{\eps})\,\varphi\diff x\diff t =  -\int_{0}^{T}\int_{\Td}m(u_{\eps})\nabla\mu_{\eps}\cdot\nabla\varphi\diff x\diff t = \\ = &-\int_{0}^{T}\int_{\Td}m(u_{\eps})\nabla(B_{\eps}[u_{\eps}] + F'(u_{\eps})) \cdot\nabla\varphi\diff x\diff t \\= & \int_{0}^{T}\int_{\Td}B_{\eps}[u_{\eps}]\, \nabla m(u_{\eps})\cdot\nabla\varphi\diff x\diff t
+\int_{0}^{T}\int_{\Td}B_{\eps}[u_{\eps}]\, m(u_{\eps})\Delta\varphi\diff x\diff t \\ &-\int_{0}^{T}\int_{\Td}m(u_{\eps})F''(u_{\eps})\nabla u_{\eps}\cdot\nabla\varphi\diff x\diff t =:I_{1}+I_{2}+I_{3}.
\end{split}
\end{equation}

We choose a subsequence as in Lemma \ref{lem:compact_ueps_collected}. Using~\eqref{eq:identities_doublewell} we can pass to the limit in $I_3$. It remains to prove convergence of $I_{1}$ and $I_{2}$. \\

\underline{\textit{Step 1: Convergence of term $I_{1}$}}. Using~\ref{propS_nonneg} in Lemma \ref{lem:S_properties} we write
\begin{align*}
I_{1}=&\int_{0}^{T}\int_{\Td}\int_{\Td}S_{\eps}[u_{\eps}] \, S_{\eps}[\nabla m(u_{\eps})\cdot\nabla\varphi]\diff x\diff y\diff t\\
=&\int_{0}^{T}\int_{\Td}\int_{\Td}S_{\eps}[u_{\eps}] \, S_{\eps}[\nabla m(u_{\eps})]\cdot\nabla\varphi\diff x\diff y\diff t+\int_{0}^{T}\int_{\Td}\int_{\Td}S_{\eps}[u_{\eps}] \, \nabla m(u_{\eps})\cdot S_{\eps}[\nabla\varphi]\diff x\diff y\diff t\\
&+R_{\eps}^{(1)} =: J_{1}^{(1)}+J_{2}^{(1)}+R_{\eps}^{(1)}, \phantom{int_{0}^{T}\int_{\Td}}
\end{align*}
where $R_{\eps}^{(1)}$ is an error defined as 
\begin{equation*}
 R_{\eps}^{(1)}= \int_{0}^{T}\int_{\Td}\int_{\Td}S_{\eps}[u_{\eps}]\left(S_{\eps}[\nabla m(u_{\eps})\cdot\nabla\varphi]-S_{\eps}[\nabla m(u_{\eps})]\cdot\nabla\varphi-\nabla m(u_{\eps})\cdot S_{\eps}[\nabla\varphi]\right)\diff x\diff y \diff t.
\end{equation*}
For the term $J_{1}^{(1)}$ we integrate by parts
$$
J_{1}^{(1)} = - \int_{0}^{T}\int_{\Td}\int_{\Td}S_{\eps}[u_{\eps}]\,S_{\eps}[m(u_{\eps})]\,\Delta \varphi\diff x\diff y\diff t
- \int_{0}^{T}\int_{\Td}\int_{\Td}S_{\eps}[m(u_{\eps})]\,S_{\eps}[\nabla u_{\eps}] \cdot \nabla \varphi\diff x\diff y\diff t
$$
For the first term we use \eqref{Sconv5} while for the second term we apply \eqref{Sconv1} to get
\begin{equation}\label{conv:J_{1}^{(1)}}
J_{1}^{(1)}\to -\frac{W}{2}\int_0^T \int_{\Td} \nabla m(u)\cdot \nabla u \, \Delta \varphi \diff x \diff t - \frac{W}{2} \int_0^T \int_{\Td} 
\nabla m(u) \cdot D^2u \nabla \varphi
\diff x \diff t.
\end{equation}
For the term $J_{2}^{(1)}$ we simply apply~\eqref{Sconv6} and obtain
\begin{equation}\label{conv:J_{2}^{(1)}}
J_{1}^{(2)}\to \f{W}{2}\int_{0}^{T}\int_{\Td}\nabla m(u) \cdot  D^{2}\varphi \nabla u \diff x\diff t.  
\end{equation}
Finally, concerning the term $R_{\eps}^{(1)}$, we will prove that it converges to 0. We first use \ref{propS_product_rule} to write
$$
R_{\eps}^{(1)}=\int_{0}^{T}\int_{\Td}\int_{\Td}\f{\omega_{\eps}(y)}{2|y|^{\alpha}\,\eps^{2-\alpha}}(u_{\eps}(x)-u_{\eps}(x-y))(\nabla m(u_{\eps})(x) -\nabla m(u_{\eps})(x-y)) \cdot (\nabla\varphi(x)-\nabla\varphi(x-y)).
$$
Integrating by parts we get that $R_{\eps}^{(1)}$ equals to:
\begin{align*}
&\int_{0}^{T}\int_{\Td}\int_{\Td}\f{\omega_{\eps}(y)}{2|y|^{\alpha}\,\eps^{2-\alpha}}(m(u_{\eps})(x)-m(u_{\eps})(x-y))(\nabla u_{\eps}(x) -\nabla u_{\eps}(x-y)) \cdot (\nabla\varphi(x)-\nabla\varphi(x-y))\\
&+ \int_{0}^{T}\int_{\Td}\int_{\Td}\f{\omega_{\eps}(y)}{2|y|^{\alpha}\,\eps^{2-\alpha}}(m(u_{\eps})(x)-m(u_{\eps})(x-y))( u_{\eps}(x) - u_{\eps}(x-y)) \, (\Delta\varphi(x)-\Delta\varphi(x-y)).
\end{align*}
We consider the first term which is harder and then we briefly comment how to handle the second one. We change variables to have
$$
\int_{0}^{T}\int_{\Td}\int_{\Td}\f{\omega(y)}{2|y|^{\alpha}\,\eps^{2}}(m(u_{\eps})(x)-m(u_{\eps})(x-\eps y))(\nabla u_{\eps}(x) -\nabla u_{\eps}(x- \eps y)) \cdot (\nabla\varphi(x)-\nabla\varphi(x-\eps y))
$$
Note that since $\{u_{\eps}\}$ is bounded, $|m(u_{\eps})(x) - m(u_{\eps})(x-y)|\leq C\,|u_{\eps}(x) - u_{\eps}(x-y)|$ so that due to \ref{item_est6} in Lemma \ref{lem:uniform_est_just_eps}, the sequence 
\begin{equation}\label{eq:energy_bound_for_m(u)}
\left\{\frac{\sqrt{\omega(y)}}{|y|^{\frac{\alpha}{2}}}\frac{|m(u_{\eps})(x) - m(u_{\eps})(x-\eps\,y)|}{\eps}\right\} \mbox{ is bounded in }
L^{\infty}(0,T; L^2(\Td\times\Td)).
\end{equation}
Moreover, $|\nabla \varphi(x) - \nabla \varphi(x- \eps\,y)|, \, |\Delta \varphi(x) - \Delta \varphi(x- \eps\,y)| \leq C\,\eps\, |y|$.
Hence, using \eqref{eq:energy_bound_for_m(u)} and \ref{item_est7} in Lemma \ref{lem:uniform_est_just_eps}, the term of interest is bounded by $C\,\varepsilon$ so it converges to 0 when $\eps \to 0$. Finally, for the second term in the definition of $R_{\eps}^{(1)}$, we argue in the same way using \eqref{eq:energy_bound_for_m(u)} and \ref{item_est6} (Lemma \ref{lem:uniform_est_just_eps}). \\

From \eqref{conv:J_{1}^{(1)}} and \eqref{conv:J_{2}^{(1)}}, using the identity
$$
\int_0^T \int_{\Td} \nabla m(u)\cdot \left[D^2 \varphi \nabla u - D^2 u \nabla \varphi \right] \diff x \diff t = \int_0^T \int_{\Td} \nabla m(u) \cdot \left[ \nabla u \, \Delta \varphi - \nabla \varphi \, \Delta u\right] \diff x \diff t,
$$
easily verifiable by integration by parts, we conclude that the limit can be simplified to
$$
I_1 \to -\frac{W}{2}\,\int_0^T \int_{\Td} \nabla m(u) \cdot \nabla \varphi \, \Delta u.  
$$

\underline{\textit{Step 2: Convergence of term $I_{2}$}}. We write
\begin{align*}
I_{2}&=\int_{0}^{T}\int_{\Td}\int_{\Td}S_{\eps}[u_{\eps}]\,S_{\eps}[m(u_{\eps})\Delta\varphi] \diff x\diff y\diff t\\
&=\int_{0}^{T}\int_{\Td}\int_{\Td}S_{\eps}[u_{\eps}]S_{\eps}[m(u_{\eps})]\Delta\varphi\diff x\diff y\diff t+\int_{0}^{T}\int_{\Td}\int_{\Td}S_{\eps}[u_{\eps}] \, m(u_{\eps}) S_{\eps}[\Delta\varphi]\diff x\diff y\diff t+R_{\eps}^{(2)}\\
&=J_{1}^{(2)}+J_{2}^{(2)}+R_{\eps}^{(2)},
\end{align*}
where $R_{\eps}^{(2)}$ equals
$$
R_{\eps}^{(2)} = \int_{0}^{T}\int_{\Td}\int_{\Td}S_{\eps}[u_{\eps}]\left(S_{\eps}[ m(u_{\eps}) \, \Delta\varphi]-S_{\eps}[ m(u_{\eps})]\, \Delta\varphi- m(u_{\eps})\, S_{\eps}[\Delta\varphi]\right)\diff x\diff y \diff t.
$$
Applying \eqref{Sconv5} and \eqref{Sconv4} we obtain that the first two terms converge to
\begin{multline*}
J_1^{(2)} + J_2^{(2)} \to \f{W}{2}\int_{0}^{T}\int_{\Td} \nabla m(u) \cdot \nabla u \, \Delta\varphi\diff x\diff t+\f{W}{2}\int_{0}^{T}\int_{\Td}m(u)\nabla u \cdot \nabla\Delta\varphi\diff x\diff t = \\ =
-\f{W}{2}\int_{0}^{T}\int_{\Td}m(u) \,\Delta u \,\Delta\varphi\diff x\diff t
\end{multline*}
so it remains to prove that $R_{\eps}^{(2)} \to 0$. In fact, using \ref{propS_product_rule} in Lemma \ref{lem:S_properties} we realize it equals
$$
\int_{0}^{T}\int_{\Td}\int_{\Td}\f{\omega_{\eps}(y)}{2|y|^{\alpha}\,\eps^{2-\alpha}}(m(u_{\eps})(x)-m(u_{\eps})(x-y))( u_{\eps}(x) - u_{\eps}(x-y)) \, (\Delta\varphi(x)-\Delta\varphi(x-y))
$$
which was already studied in the case of $R_{\eps}^{(1)}$. We obtain that
$$
I_2 \to -\f{W}{2}\int_{0}^{T}\int_{\Td}m(u) \,\Delta u \,\Delta\varphi\diff x\diff t
$$
and the proof is concluded.
\end{proof}

\section{Numerical simulations}\label{sect:numerics}
To illustrate the behavior of the solutions as $\eps\to 0$, we propose numerical simulations. To keep things simple, we assume that the mobility is constant, $m(u)=1$, and we use a smooth approximation of the double-well potential $F(u)= u^{2}(1-u)^{2}$ in Equations~\eqref{eq:CHE1}-\eqref{eq:CHE2} and~\eqref{eq:CH1}-\eqref{eq:CH2}. We refer to~\cite{MR4317553}, where authors investigate how the non-smooth nature of the double-well potential (such as the double-obstacle potential) generates sharp interfaces.\\

Our simulations are made in both one and two dimensions. The kernel reads
$$
J_{\eps}(x)=\f{1}{\eps^{d+2}}\exp\left(-\f{|x|^{2}}{\eps^{2}}\right).
$$

We consider three specific scenarios: $\eps=1$, $\eps=0.7$, and $\eps=0.1$. In the one-dimensional case, we superimpose the curve of the nonlocal solution (represented in red) onto the curve of the local solution (depicted in blue). We observe their evolution over time at $t=1$, $t=10$, and $t=100$.\\

In the two-dimensional case, the first row showcases the evolution of the nonlocal Cahn-Hilliard equation at $t=10$, $t=100$, and $t=1000$. The last row represents the evolution of the local equation.\\

The numerical implementation was made using Python, and numpy.fft, numpy.ifft to efficiently compute the Fourier and inverse Fourier transforms. The Fourier transform is particularly useful for terms like $\Delta^2 u$
and the convolution $\Delta (u-J\ast u)$. \\

Focusing on the numerical results, the observations made in both one and two dimensions are the same. As $\eps$ is sent to zero, the solutions of the non-local equation converge to those of the local equation. When $\eps$ is set to 0.1, the difference becomes invisible to the naked eye.

We opted for a random initial condition with low regularity. Within the simulations for both one-dimensional and two-dimensional scenarios with $\eps=1$, we observe two interesting aspects. First, the regularity of solutions to the nonlocal equation is lower than that of the local one and the initial condition requires a longer time to be smoothed. From a theoretical point of view, this can be explained by the $H^1$ regularity (a porous medium character) for the nonlocal equation compared to the $H^2$ regularity for the local equation. The second point of interest lies in the formation of an interface in a long time which is of different nature for both equations. For, $\eps=1$, at the final time step, the interface is sharp, unlike the diffuse interface observed for the local equation.

\begin{figure}[htbp]
  \centering
  \begin{minipage}{0.33\textwidth}
    \centering
    \includegraphics[width=\linewidth]{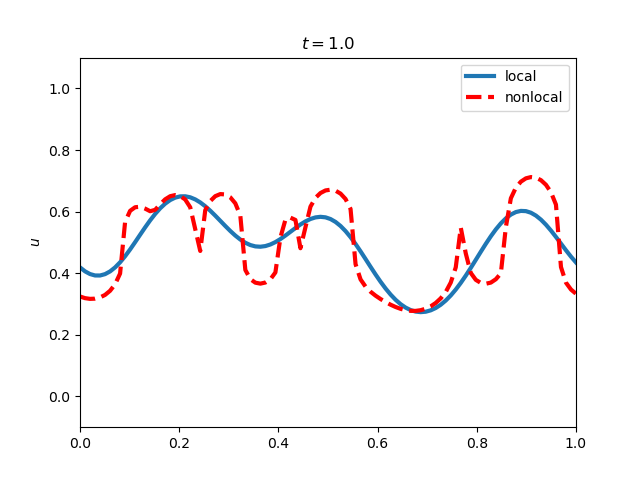}
    \caption*{$\varepsilon=1$, $t=1$}
    \label{fig:image1}
  \end{minipage}%
  \hfill
  \begin{minipage}{0.33\textwidth}
    \centering
    \includegraphics[width=\linewidth]{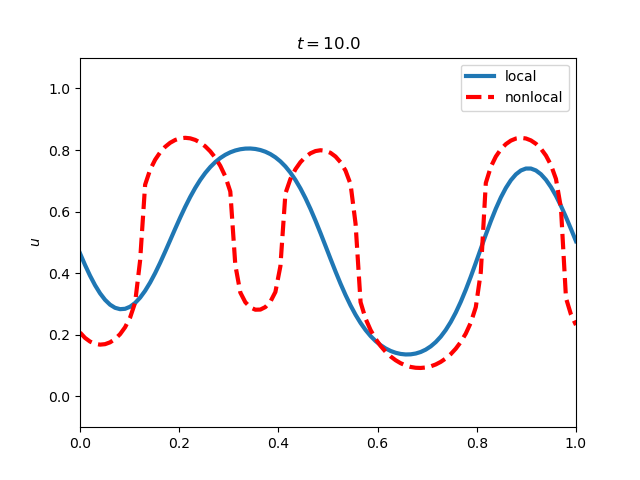}
    \caption*{$\varepsilon=1$, $t=10$}
    \label{fig:image2}
  \end{minipage}%
  \hfill
  \begin{minipage}{0.33\textwidth}
    \centering
    \includegraphics[width=\linewidth]{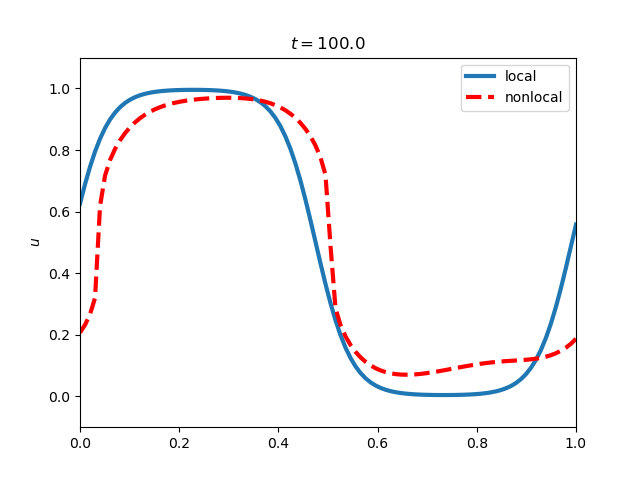}
    \caption*{$\varepsilon=1$, $t=100$}
    \label{fig:image3}
  \end{minipage}
  
  \vspace{1cm}
  
  \begin{minipage}{0.33\textwidth}
    \centering
    \includegraphics[width=\linewidth]{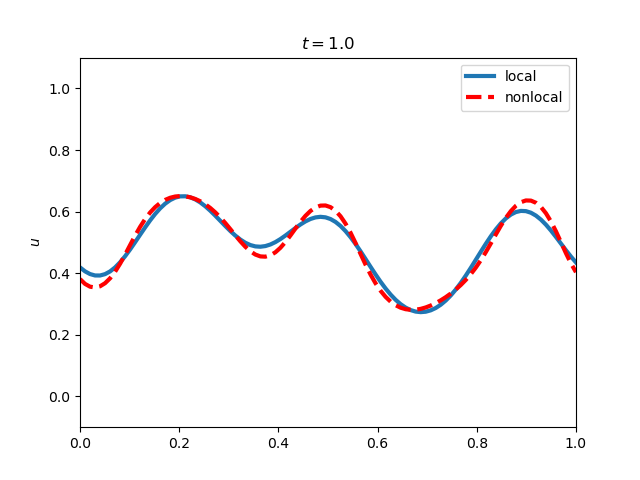}
    \caption*{$\varepsilon=0.7$, $t=1$}
    \label{fig:image4}
  \end{minipage}%
  \hfill
  \begin{minipage}{0.33\textwidth}
    \centering
    \includegraphics[width=\linewidth]{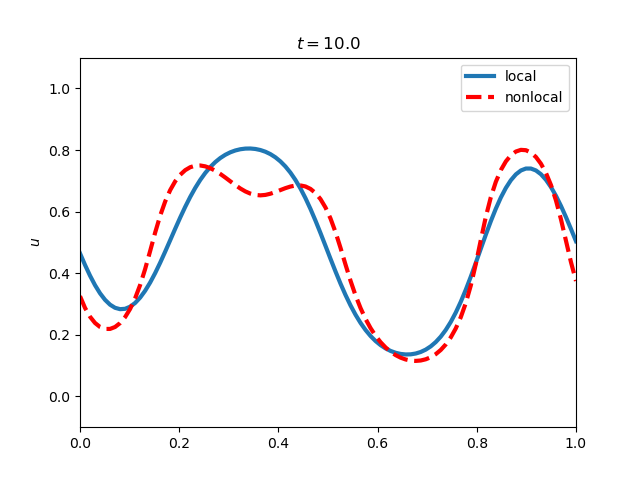}
    \caption*{$\varepsilon=0.7$, $t=10$}
    \label{fig:image5}
  \end{minipage}%
  \hfill
  \begin{minipage}{0.33\textwidth}
    \centering
    \includegraphics[width=\linewidth]{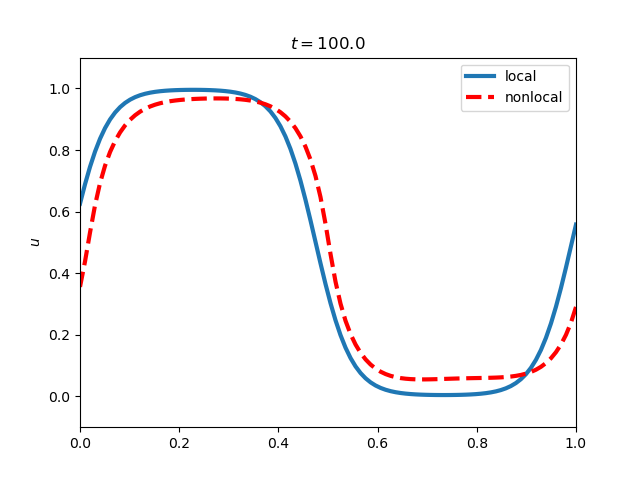}
    \caption*{$\varepsilon=0.7$, $t=100$}
    \label{fig:image6}
  \end{minipage}
  
  \vspace{1cm}
  
  \begin{minipage}{0.33\textwidth}
    \centering
    \includegraphics[width=\linewidth]{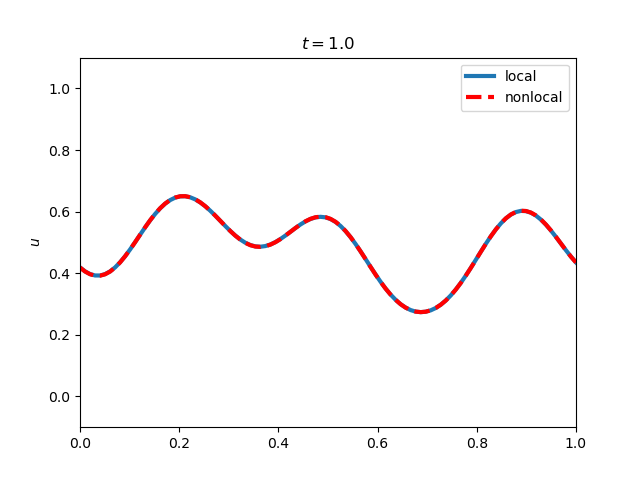}
    \caption*{$\varepsilon=0.1$, $t=1$}
    \label{fig:image7}
  \end{minipage}%
  \hfill
  \begin{minipage}{0.33\textwidth}
    \centering
    \includegraphics[width=\linewidth]{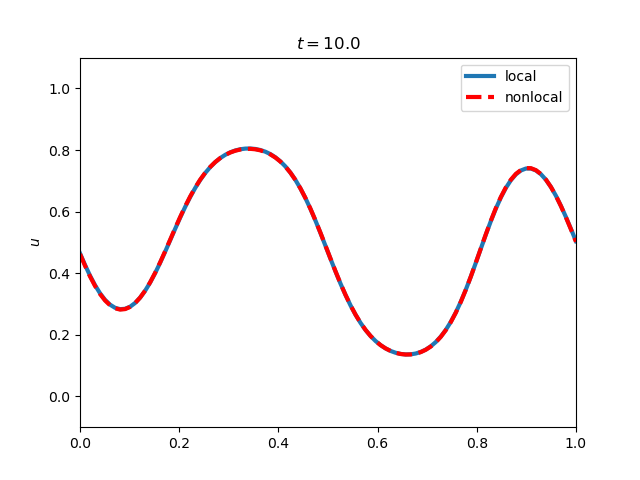}
    \caption*{$\varepsilon=0.1$, $t=10$}
    \label{fig:image8}
  \end{minipage}%
  \hfill
  \begin{minipage}{0.33\textwidth}
    \centering
    \includegraphics[width=\linewidth]{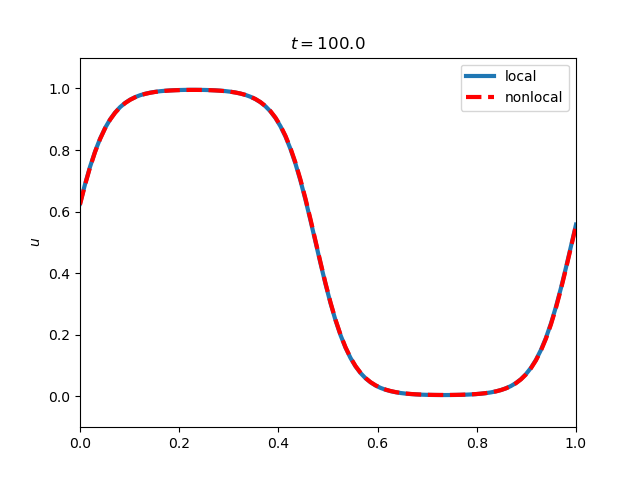}
    \caption*{$\varepsilon=0.1$, $t=100$}
    \label{fig:image9}
  \end{minipage}
\end{figure}

\begin{figure}[htbp]
  \centering
  \begin{minipage}{0.33\textwidth}
    \centering
    \includegraphics[width=\linewidth]{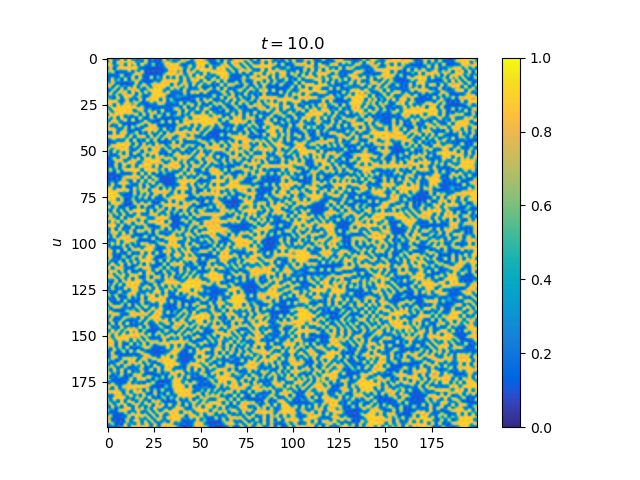}
    \caption*{$\varepsilon=1$, $t=10$}
  \end{minipage}%
  \hfill
  \begin{minipage}{0.33\textwidth}
    \centering
    \includegraphics[width=\linewidth]{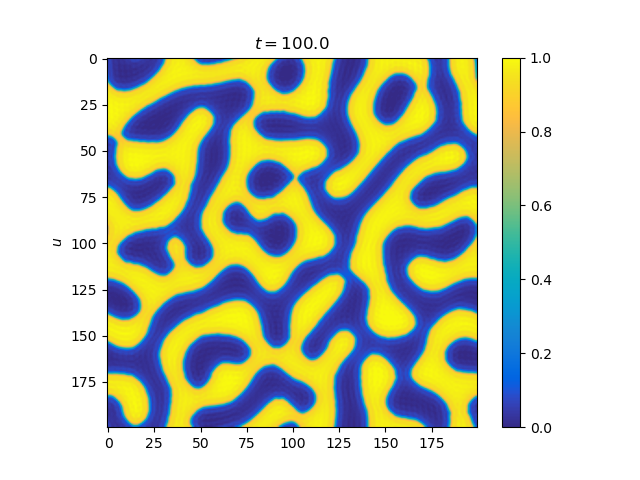}
    \caption*{$\varepsilon=1$, $t=100$}
  \end{minipage}%
  \hfill
  \begin{minipage}{0.33\textwidth}
    \centering
    \includegraphics[width=\linewidth]{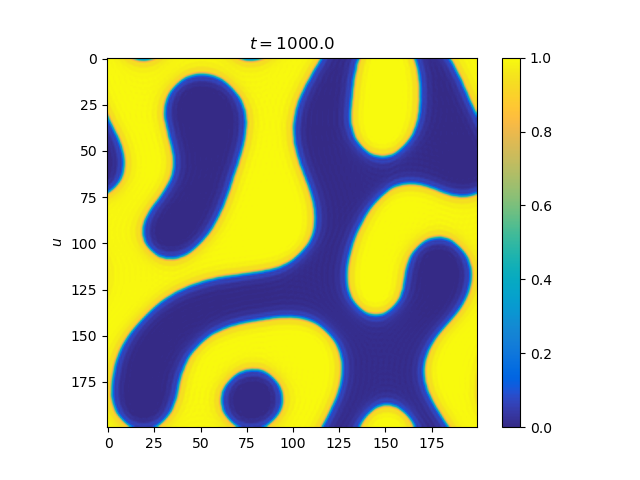}
    \caption*{$\varepsilon=1$, $t=1000$}
  \end{minipage}
  
  \vspace{1cm}  
  \begin{minipage}{0.33\textwidth}
    \centering
    \includegraphics[width=\linewidth]{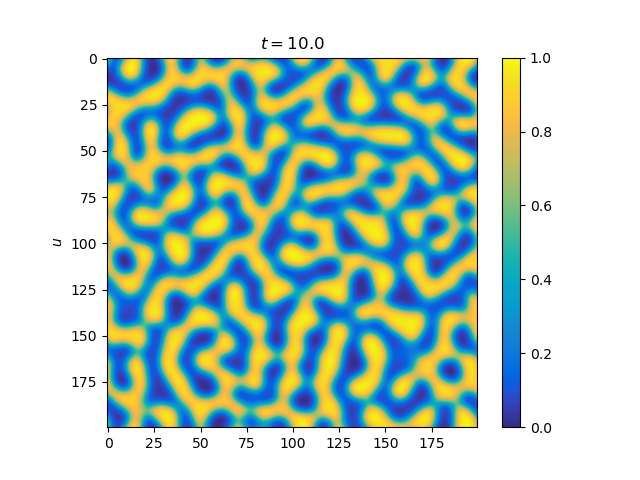}
    \caption*{local, $t=10$}
  \end{minipage}%
  \hfill
  \begin{minipage}{0.33\textwidth}
    \centering
    \includegraphics[width=\linewidth]{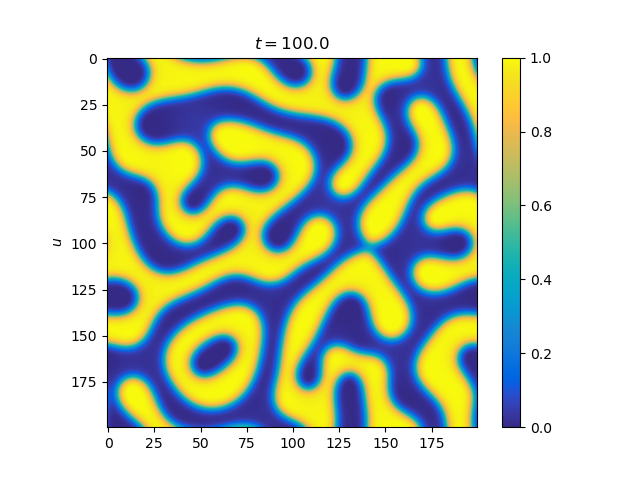}
    \caption*{local, $t=10$}
  \end{minipage}%
  \hfill
  \begin{minipage}{0.33\textwidth}
    \centering
    \includegraphics[width=\linewidth]{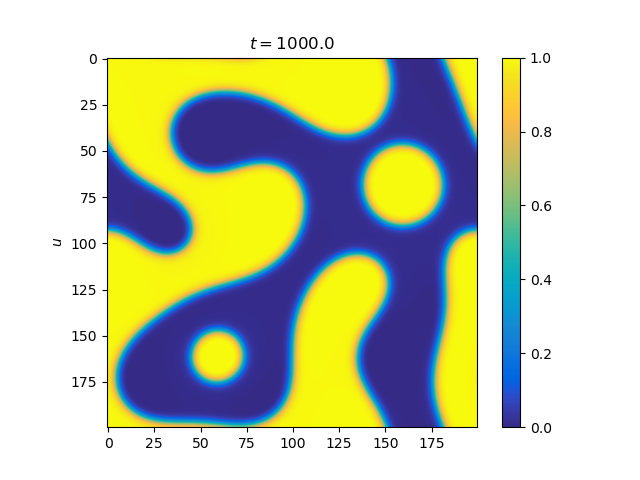}
    \caption*{local, $t=1000$}
  \end{minipage}
\end{figure}

\section{The nonlocal-to-local limit for the cell-cell adhesion model}\label{sect:cell_cell}
Here, we explain, using similar ideas, how to treat a recently introduced model in the theory of cell-cell adhesion from \cite[eq. (15)]{MR3948738} which also involves the mobility $m(u) = u\,(1-u)$. More precisely, we consider the PDE posed on $[0,T]\times \Td$ 
\begin{equation}\label{eq:PDE_cell_cell_adhesion_noeps!}
\partial_t u = \DIV (u \nabla u) - a\,\DIV( u\, (1-u)\, K[u]),
\end{equation}
where the parameter $a$ measures the relative cell-cell adhesion strength while $K$ is the nonlocal kernel $K$ defined as
$$
K[u] = \int_0^1 \int_{S^{d-1}} u(x + r\, \eta) \, V(r\,\eta)\, r^{d-1} \, \eta \diff S(\eta) \, \diff r.
$$
Here, $V$ is a smooth, radially symmetric and compactly supported kernel with $\supp V \subset \Td$ and $S^{d-1}$ is a unit sphere in $\R^d$. It is useful to observe that, with the assumptions on $V$, $K$ can be written as a convolution operator
$$
K[u](x) = \int_{\Td} u(x-y)\, V(y)\,\frac{y}{|y|} \diff y  = u \ast \left(V(\cdot)\, \left(\frac{\cdot}{|\cdot|}\right) \right)
$$

We now rescale the kernel by taking $V(x) := \frac{1}{\eps} \, \omega_{\eps}(x)$ in \eqref{eq:PDE_cell_cell_adhesion_noeps!} where $\omega_{\eps}(x) = \frac{1}{\eps^d} \omega\left(\frac{x}{\eps} \right)$ and $\omega: \R^d \to \R$ is a smooth, radially symmetric and compactly supported kernel. The scaling factor $\frac{1}{\eps}$ is chosen to obtain a meaningful PDE in the limit $\eps \to 0$ as explained below. Hence, we consider solutions to the PDE
\begin{equation}\label{eq:PDE_cell_cell_adhesion}
\partial_t \ueps = \DIV (\ueps \nabla \ueps) - a\,\DIV( \ueps\, (1-\ueps)\, K_{\eps}[\ueps])
\end{equation}
with 
\begin{equation}\label{eq:kernel_K_eps_def_cell_cell}
K_{\eps}[u] =  \frac{1}{\eps} \int_{\Td} u(x-y)\, \omega_{\eps}(y)\,\frac{y}{|y|} \diff y
\end{equation}
and we want to understand the limit $\eps \to 0$. To guess what is the expected PDE, let $\psi \in C_c^{\infty}((0,T)\times\Td)$. Then, by Taylor's expansion,
\begin{equation}\label{eq:Taylor_K_eps_cell_cell}
\begin{split}
K_{\eps}[\psi] &= \frac{1}{\eps} \int_{\Td} (\psi(x-y) - \psi(x))\, \omega_{\eps}(y)\, \frac{y}{|y|} \diff y = \\ &= \nabla \psi(x)\, \frac{1}{\eps} \int_{\Td}  \omega_{\eps}(y)\, |y| \diff y + R_{\eps} = \nabla \psi(x) \, \left(\int_{\Td} \omega(y)\,|y| \diff y \right) + R_{\eps},
\end{split}
\end{equation}
where $R_{\eps} \leq C(\omega)\, \|D^2 \psi\|_{\infty}\, \eps \to 0$ when $\eps \to 0$. Hence, we expect that the limiting PDE reads
\begin{equation}\label{eq:local_cell_cell_adhesion}
\partial_t u = \DIV (u \nabla u) - C_{\omega}\,a\,\DIV( u\, (1-u)\, \nabla u),
\end{equation}
where $C_{\omega}:= \int_{\Td} \omega(y)\,|y| \diff y$.\\

To have sufficient uniform a priori estimates with respect to $\eps \in (0,1)$, we have to assume that the parameter $a$ satisfies
\begin{equation}\label{ass:cell-cell:parameter_a}
|a|\, \sqrt{C} < 1,
\end{equation}
where $C$ is a constant in the inequality
\begin{equation}\label{eq:Poincare_rev_special_case_cell_cell_adh}
\int_{\Td} \int_{\Td} \frac{|f(x)- f(y)|^2}{\eps^{2}} \omega_{\varepsilon}(x-y) \diff x \diff y \leq C \, \|f\|_{H^1(\Td)}^2,
\end{equation}
which holds true by Lemma \ref{lem:inv_poincare_ineq}. The inequality \eqref{eq:Poincare_rev_special_case_cell_cell_adh} implies also the uniform bound on $K_{\eps}[u]$: using that $\int_{\Td} \omega_{\eps}(y) \frac{y}{|y|} \diff y =0$ we have
$$
K_{\eps}[u] = \frac{1}{\eps} \int_{\Td} u(x-y) \, \omega_{\eps}(y)\, \frac{y}{|y|} \diff y = \frac{1}{\eps} \int_{\Td} (u(x-y) - u(x))\, \omega_{\eps}(y)\, \frac{y}{|y|} \diff y
$$
so that by Jensen's inequality
\begin{equation}\label{eq:estimate_K_in_L2}
\| K_{\eps}[u] \|_{L^2(\Td)}^2 \leq \int_{\Td} \int_{\Td} \frac{(u(x-y) - u(x))^2}{\eps^2}\, \omega_{\eps}(y) \diff y \diff x \leq C\, \| \nabla u\|_{L^2(\Td)}^2.
\end{equation} 

For the nonlocal-to-local limit, the main idea is again the same as in the case of the Cahn-Hilliard system: the mobility $m(u) = u\,(1-u)$ ensures that the solution remains bounded so that we can easily handle the nonlinear terms. We also want to comment that in \cite{MR3948738} Authors derive also systems of PDEs of the form \eqref{eq:PDE_cell_cell_adhesion} but these are not treatable mathematically because, in full generality, there is currently no mathematical theory of such systems. Here we focus on a single equation.

\subsection{Existence of weak solutions to \eqref{eq:PDE_cell_cell_adhesion}}

\begin{definition}\label{def:weak_sol_cell_cell_eps}
We say that a function $u_{\eps}: [0,T]\times\Td \to [0,1]$ is a weak solution to \eqref{eq:PDE_cell_cell_adhesion} with initial condition $u^0$ if $\nabla u_{\eps} \in L^2((0,T)\times\Td)$ and for all $\psi \in C_c^{\infty}([0,T)\times \Td)$ we have
$$
\int_{\Td} u^0\, \psi(0,x) \diff x +\int_0^T \int_{\Td} u_{\eps}\, \partial_t \psi \diff x \diff t= \int_0^T \int_{\Td} (u_{\eps}\, \nabla \ueps - a\,(1-u_{\eps})\, u_{\eps} \, K_{\eps}[u_{\eps}]) \cdot \nabla \psi \diff x \diff t. 
$$
\end{definition}
\begin{lemma}\label{lem:existence_cell_cell_adh}
Let $0 \leq u^0 \leq 1$ and let the parameter $a$ satisfy \eqref{ass:cell-cell:parameter_a}. Then, there exists a weak solution to \eqref{eq:PDE_cell_cell_adhesion} with the initial condition $u^0$. Moreover, the following sequences are bounded uniformly with respect to $\eps \in (0,1)$:
\begin{enumerate}[label=(E\arabic*)]
\item\label{estim_eps_cell_eps_01} $\{ \ueps\}$ in $L^{\infty}((0,T)\times\Td)$ with $0 \leq u_{\eps} \leq 1$,
\item $\{\nabla \ueps \}$ in $L^2((0,T)\times\Td)$,
\item\label{estim_eps_cell_eps_time_der} $\{\partial_t \ueps \}$ in $L^2(0,T; H^{-1}(\Td))$.
\end{enumerate}
\end{lemma}
\begin{proof}
\textit{\underline{Step 1: An auxiliary problem.}} For fixed $\eps \in (0,1)$ we consider an auxiliary problem
\begin{equation}\label{eq:auxillary_pr_cell-cell}
\partial_t u_{\delta} = \DIV (u_{\delta} \nabla u_{\delta}) - a\,\DIV( (1-u_{\delta})\, (u_{\delta}\, K_{\eps}[u_{\delta}])\ast \varphi_{\delta})
\end{equation}
where $\{\varphi_{\delta}\}_{\delta\in(0,1)}$ is a standard mollifying sequence. A solution to \eqref{eq:auxillary_pr_cell-cell} with regularity $u_{\delta} \in L^2(0,T; H^1(\Td))$ can be obtained by standard methods, for example by the Schauder's fixed point theorem. We want to pass with $\delta \to 0$ in \eqref{eq:auxillary_pr_cell-cell}.\\

\textit{\underline{Step 2: Uniform bounds with respect to $\delta$ and $\eps$.}} We claim that the following sequences are uniformly bounded with respect to $\delta, \eps \in (0,1)$:
\begin{enumerate}[label=(F\arabic*)]
\item\label{est:u_0_1_delta_before_eps} $\{ u_{\delta}\}$ in $L^{\infty}((0,T)\times\Td)$ with $0 \leq u_{\delta} \leq 1$,
\item\label{est:delta_gradient} $\{\nabla u_{\delta}\}$ in $L^2((0,T)\times\Td)$,
\item\label{est:time_der_delta} $\{\partial_t u_{\delta}\}$ in $L^2(0,T; H^{-1}(\Td))$.
\end{enumerate}
Let $p$ be a function such that $p(s) = 0$ for $s<1$ and $p'(s) \geq 0$ for $s \geq 1$. Let $P(s) = \int_1^s p(r) \diff r$. Integrating in space we deduce
$$
\partial_t \int_{\Td} P(u_{\delta}) \diff x = - \int_{\Td} u_{\delta} \, p'(u_{\delta}) \, |\nabla u_{\delta}|^2 \diff x - a\, \int_{\Td} p(u_{\delta})\, \DIV( (1-u_{\delta})\, (u_{\delta}\, K_{\eps}[u_{\delta}])\ast \varphi_{\delta}) \diff x.
$$
Since $ \int_{\Td} u_{\delta} \, p'(u_{\delta}) \, |\nabla u_{\delta}|^2 \diff x \geq 0$ and $u_{\delta} \in L^2(0,T; H^1(\Td))$, we can approximate function $\mbox{sign}_+(u-1) = \mathds{1}_{u>1}$ with $p$ and deduce
$$
\partial_t \int_{\Td} |u_{\delta}-1|_+ \diff x \leq - a\, \int_{\Td} \mbox{sign}_+(u_{\delta}-1)\, \DIV( (1-u_{\delta} ) \, (u_{\delta}\, K_{\eps}[u_{\delta}])\ast \varphi_{\delta}) \diff x,
$$
where $|u-1|_+ = (u-1)\,\mathds{1}_{u>1}$. The integrand on the (RHS) can be written again in the divergence form as $-\DIV( (1-u_{\delta} )_+ \, (u_{\delta}\, K_{\eps}[u_{\delta}])\ast \varphi_{\delta})$ so the respective integral vanish and we conclude that $u_{\delta} \leq 1$. In a similar way we prove that $u_{\delta} \geq 0$. \\

\noindent Next, we multiply by $\log u_{\delta}$ and integrate by parts to get
\begin{equation}\label{eq:energy_estimate_cell_cell_adhesion_delta}
\begin{split}
\partial_t \int_{\Td} u_{\delta} \log u_{\delta} \diff x + &\int_{\Td} |\nabla u_{\delta}|^2 \diff x \leq |a| \, \int_{\Td} |\nabla u_{\delta}| \, |u_{\delta}\, K_{\eps}[u_{\delta}]|\ast \varphi_{\delta} \diff x  \leq \\  \phantom{\int_{\Td}} &\leq |a|\, \| \nabla u_{\delta} \|_{L^2(\Td)} \, \|u_{\delta}\, K_{\eps}[u_{\delta}] \|_{L^2(\Td)} \leq 
|a|\, \| \nabla u_{\delta} \|_{L^2(\Td)} \, \|K_{\eps}[u_{\delta}] \|_{L^2(\Td)}.
\end{split}
\end{equation}
Using \eqref{eq:estimate_K_in_L2}, we obtain
$$
\partial_t \int_{\Td} u_{\delta} \log u_{\delta} \diff x + \int_{\Td} |\nabla u_{\delta}|^2 \diff x \leq |a| \, \int_{\Td} |\nabla u_{\delta}| \, |u_{\delta}\, K_{\eps}[u_{\delta}]|\ast \varphi_{\delta} \diff x  \leq |a|\, \sqrt{C} \,  \| \nabla u_{\delta}\|_{L^2(\Td)}^2.
$$
Hence, since $|a|\, \sqrt{C}<1$, we obtain \ref{est:delta_gradient}.\\

Finally, multiplying \eqref{eq:auxillary_pr_cell-cell} by $\psi \in C^{\infty}_c((0,T)\times \Td)$ and using that $0 \leq u_{\delta} \leq 1$ as well as \eqref{eq:estimate_K_in_L2} we have
\begin{align*}
\left|\int_0^T \int_{\Td} u_{\delta}\, \partial_t \psi \diff x \diff t \right| &\leq \left|\int_0^T \int_{\Td} u_{\delta}\, \nabla u_{\delta}\,  \nabla \psi \diff x \diff t \right| + \left|\int_0^T \int_{\Td}(1-u_{\delta})\, (u_{\delta}\, K_{\eps}[u_{\delta}])\ast \varphi_{\delta} \,  \nabla \psi \diff x \diff t \right| \\
&\leq \left(\|\nabla u_{\delta} \|_{L^2((0,T)\times\Td)} + \| K_{\eps}[u_{\delta}] \|_{L^2((0,T)\times\Td)} \right) \, \| \nabla \psi \|_{L^2((0,T)\times\Td)}   \phantom{\int_{\Td}}\\
& \leq (1+\sqrt{C})\, \|\nabla u_{\delta} \|_{L^2((0,T)\times\Td)} \, \| \nabla \psi \|_{L^2((0,T)\times\Td)},  \phantom{\int_{\Td}}
\end{align*}
so that \ref{est:time_der_delta} follows by \ref{est:delta_gradient}. \\

\textit{\underline{Step 3: The limit $\delta \to 0$ and the estimates \ref{estim_eps_cell_eps_01}--\ref{estim_eps_cell_eps_time_der}.}} By the Aubin-Lions lemma, the Banach-Alaoglu theorem and interpolation in $L^p$ spaces, there exists a subsequence such that 
\begin{equation}\label{eq:convergence_delta_cell_cell}
\begin{split}
&u_{\delta} \to u \mbox{ strongly in } L^p((0,T)\times \Td) \mbox{ for } p \in [1, \infty),\\
&\nabla u_{\delta} \rightharpoonup u \mbox{ weakly in } L^2((0,T)\times\Td),\\
&\partial_t u_{\delta} \rightharpoonup \partial_t u \mbox{ weakly in } L^2(0,T; H^{-1}(\Td)).
\end{split}
\end{equation}
This is sufficient to pass to the limit $\delta\to0$ in the expression
$$
\int_{\Td} u^0\, \psi(0,x) \diff x +\int_0^T \int_{\Td} u_{\delta}\, \partial_t \psi \diff x \diff t= \int_0^T \int_{\Td} (u_{\delta}\, \nabla u_{\delta} - a\,(1-u_{\delta})\, (u_{\delta} \, K_{\eps}[u_{\delta}])\ast\varphi_{\delta}) \cdot \nabla \psi \diff x \diff t 
$$
(keep in mind that $\eps$ is fixed!) for $\psi \in C^{\infty}_c([0,T); \Td)$ so that $u$ satisfies Definition \ref{def:weak_sol_cell_cell_eps}. Finally, since bounds \ref{est:u_0_1_delta_before_eps}--\ref{est:time_der_delta} are preserved along the limit \eqref{eq:convergence_delta_cell_cell}, \ref{estim_eps_cell_eps_01}--\ref{estim_eps_cell_eps_time_der} follow.

\end{proof}

\subsection{Passing to the limit $\eps\to0$ in \eqref{eq:PDE_cell_cell_adhesion}.}

\begin{theorem}
Let $0 \leq u^0 \leq 1$ and let the parameter $a$ satisfy \eqref{ass:cell-cell:parameter_a}. Let $\{u_{\eps}\}_{\eps \in (0,1)}$ be a sequence of solutions to \eqref{eq:PDE_cell_cell_adhesion} constructed in Lemma \ref{lem:existence_cell_cell_adh} satisfying uniform bounds \ref{estim_eps_cell_eps_01}--\ref{estim_eps_cell_eps_time_der}. Then, there exists a subsequence (not relabelled) such that
\begin{equation}\label{eq:convergence_eps_cell_cell}
\begin{split}
&u_{\eps} \to u \mbox{ strongly in } L^p((0,T)\times \Td) \mbox{ for } p \in [1, \infty),\\
&\nabla u_{\eps} \rightharpoonup u \mbox{ weakly in } L^2((0,T)\times\Td),
\end{split}
\end{equation}
where $u$ satisfies 
$$
\int_{\Td} u^0\, \psi(0,x) \diff x +\int_0^T \int_{\Td} u\, \partial_t \psi \diff x \diff t= \int_0^T \int_{\Td} (u\, \nabla u - C_{\omega}\,a\,(1-u)\, u \, \nabla u) \cdot \nabla \psi \diff x \diff t, 
$$
for all $\psi \in C^{\infty}_c([0,T); \Td)$, i.e. $u$ is a distributional solution of \eqref{eq:local_cell_cell_adhesion}.
\end{theorem}
\begin{proof}
As in the proof of Lemma \ref{lem:existence_cell_cell_adh}, by the Aubin-Lions lemma, the Banach-Alaoglu theorem, interpolation in $L^p$ spaces and the bounds \ref{estim_eps_cell_eps_01}--\ref{estim_eps_cell_eps_time_der}, there exists a subsequence as in \eqref{eq:convergence_eps_cell_cell}.\\

We want to pass to the limit $\eps \to 0$ in Definition \ref{def:weak_sol_cell_cell_eps}. The only difficulty is to justify the limit in the term
\begin{equation}\label{eq:passing_to_the_limit_difficult_term}
\int_0^T \int_{\Td} (1-u_{\eps})\, u_{\eps} \, K_{\eps}[u_{\eps}] \cdot \nabla \psi \diff x \diff t
\end{equation}
due to the singularity $\frac{1}{\eps}$ in the nonlocal operator $K_{\eps}[u_{\eps}]$. By the bound \eqref{eq:estimate_K_in_L2}, there exists $\xi$ such that
$$
K_{\eps}[u_{\eps}] \rightharpoonup \xi \mbox{ weakly in } L^2((0,T)\times \Td).
$$
To identify $\xi$, let $\psi \in C_c^{\infty}((0,T)\times\Td)$. Since $K_{\eps}$ is a convolution operator with an antisymmetric kernel (see \ref{eq:kernel_K_eps_def_cell_cell})
\begin{equation}\label{eq:identity_to_identify_xi_before_passing_eps_cell_cell}
\int_0^T \int_{\Td} K_{\eps}[u_{\eps}] \, \psi \diff x \diff t = - \int_0^T \int_{\Td} u_{\eps} \, K_{\eps}[\psi] \diff x \diff t.
\end{equation}
By \eqref{eq:Taylor_K_eps_cell_cell}, $K_{\eps}[\psi] \to C_{\omega}\,\nabla \psi$ so that passing to the limit in \eqref{eq:identity_to_identify_xi_before_passing_eps_cell_cell} we obtain
$$
\int_0^T \int_{\Td} \xi \, \psi \diff x \diff t = - C_{\omega}\, \int_0^T \int_{\Td} u\, \nabla \psi  \diff x \diff t
$$
which implies $\xi = C_{\omega}\, \nabla u$. Furthermore, by \eqref{eq:convergence_eps_cell_cell},
$$
u_{\eps}\, (1-u_{\eps}) \to u\,(1-u) \mbox{ strongly in }  L^2((0,T)\times\Td)
$$
so that we can pass to the limit in \eqref{eq:passing_to_the_limit_difficult_term} 
$$
\lim_{\eps\to 0} \int_0^T \int_{\Td} (1-u_{\eps})\, u_{\eps} \, K_{\eps}[u_{\eps}] \cdot \nabla \psi \diff x \diff t = C_{\omega} \int_0^T \int_{\Td} (1-u)\,u\, \nabla u \cdot \nabla \psi \diff x \diff t 
$$
and the proof is concluded.
\end{proof}
\section*{Acknowledgements}
Jakub Skrzeczkowski was supported by the Advanced Grant Nonlocal-CPD (Nonlocal PDEs for Complex Particle Dynamics: Phase Transitions, Patterns and Synchronization) of the European Research Council Executive Agency (ERC) under the European Union’s Horizon 2020 research and Innovation programme (grant agreement No. 883363).

\appendix
\section{Results from classical analysis}

\begin{lemma}\label{lem:diff_quot_strong_conv}
Let $\{u_\eps\}_{\eps}$ be a sequence strongly compact in $L^2(0,T; H^1(\Td))$. Then, 
$$
\frac{u_{\eps}(t,x-\eps y) - u_{\eps}(t,x)}{\eps|y|} \to - \nabla u(t,x) \cdot \f{y}{|y|} \mbox{ strongly in } L^{\infty}_{y}(\Td; L^2_{(t,x)}((0,T)\times\Td)).
$$
\end{lemma}

\begin{proof}
We write 
\begin{align*}
\frac{u_{\eps}(t,x-\eps y) - u_{\eps}(t,x)}{\eps|y|}& = -\int_{0}^{1}\nabla u_{\eps}(t,x-\eps s y) \cdot \f{y}{|y|} \diff s\\
&= -\nabla u_{\eps}(t,x) \f{y}{|y|} - \int_{0}^{1}[\nabla u_{\eps}(t,x-\eps s y) - \nabla u_{\eps}(t,x)] \cdot \f{y}{|y|} \diff s.
\end{align*}

Therefore 
\begin{align*}
\left\|\frac{u_{\eps}(t,x-\eps y) - u_{\eps}(t,x)}{\eps|y|}+\nabla u(t,x) \f{y}{|y|}\right\|_{L^{2}_{t,x}}& \le \left\|(\nabla u_{\eps}-\nabla u)\cdot\f{y} {|y|}\right\|_{L^{2}_{t,x}} \\&+ \int_{0}^{1}\int_{(0,T)\times\Td} |\nabla u_{\eps}(t,x-\eps s y) - \nabla u_{\eps}(t,x) |^2\diff t\diff x\diff s
\end{align*}

With the assumptions of the lemma, we deduce that the first term converges to 0 uniformly with respect to $y$. Concerning the second term, we use the Fréchet-Kolmogorov theorem to deduce that the translation converge to 0 uniformly: there exists a modulus of continuity $\gamma:[0,\infty) \to [0,\infty)$ such that $\lim_{r\to 0} \gamma(r) = 0$ and
$$
\int_{(0,T)\times\Td} |\nabla u_{\eps}(t,x-\eps s y) - \nabla u_{\eps}(t,x) |^2\diff x\diff t \leq \gamma(\eps\, s \, |y|) \leq \gamma(\eps).
$$
This proves the result. 
\end{proof}

\section{Bourgain-Brézis-Mironescu and Ponce compactness result}
We consider a sequence of radial functions $\{\rho_{\eps}\}_{\eps}$ such that $\rho_{\eps} \geq 0$, $\int_{\R^d} \rho_{\eps} \diff x = 1$ and 
$$
\lim_{\eps \to 0} \int_{|x| > \delta} \rho_{\eps}(x) \diff x = 0 \mbox{ for all } \delta > 0.
$$
For the formulation of the compactness result, we use another sequence $\{\varphi_{\delta}\}_{\delta \in (0,1)} \subset C_{c}^{\infty}(\R^{d})$ of standard mollifiers with mass~1 such that $\varphi_{\delta}(x)=\frac{1}{\delta^{d}}\varphi(\frac{x}{\delta})$ with $\varphi$ of mass 1 and compactly supported.

\begin{theorem}\label{thm:ponce_tx}
Let $d \geq 2$. Let $\{f_\eps\}$ be a sequence bounded in $L^2((0,T)\times \Td)$. Suppose that there exists a sequence $\{\rho_{\eps}\}$ as above such that
\begin{equation}\label{eq:Ponce_org_condition}
\int_0^T \int_{\Td} \int_{\Td} \frac{|f_\eps(t,x) - f_\eps(t,y)|^2}{|x-y|^2} \rho_\eps(|x-y|) \diff x \diff y \diff t \leq C
\end{equation}
for some constant $C$. Then, $\{f_\eps\}$ is compact in space in $L^2((0,T)\times \Td)$, i.e.
\begin{equation}\label{eq:equicontinuity_Lp_space}
\lim_{\delta \to 0} \limsup_{\varepsilon \to 0} \int_0^T \int_{\Td} |f_\eps \ast \varphi_{\delta}(t,x) - f_\eps(t,x)|^2 \diff x \diff t = 0.
\end{equation}
\end{theorem}

\begin{remark}\label{rem:adapt_ponce}
Let $\omega:\R^d \to \R$ be a function as in~\eqref{eq:omega}-\eqref{as:omega}. Consider $\omega_{\eps} = \frac{1}{\eps^d} \omega\left(\frac{x}{\eps}\right)$. Suppose that
$$
\int_0^T \int_{\Td} \int_{\Td} \frac{|f_\eps(x) - f_\eps(y)|^2}{\eps^{2-\alpha}|y|^{\alpha}} \omega_\eps(|x-y|) \diff x \diff y \diff t \leq \widetilde{C}.
$$
Then, \eqref{eq:Ponce_org_condition} is satisfied. Indeed, we consider
\begin{equation}\label{eq:new_rescaled_kernel}
\rho_{\eps}(x) = \frac{\omega_\eps(|x|)\, |y|^{2-\alpha}}{\eps^{2-\alpha} \, \int_{\R^d} \omega(y) |y|^{2-\alpha} \diff y}
\end{equation}
so that \eqref{eq:Ponce_org_condition} holds true with $\frac{\widetilde{C}}{\int_{\R^d} \omega(y) |y|^{2-\alpha} \diff y}$.
\end{remark}

\section{Nonlocal Poincaré inequalities}

Let $\omega:\R^d \to \R$ be a smooth function, supported in the unit ball such that $\int_{\R^d} \omega(x) \diff x = 1$. Consider $\omega_{\eps} = \frac{1}{\eps^d} \omega\left(\frac{x}{\eps}\right)$.

\begin{lemma}\label{lem:Poincare_with_average}
For all $\alpha \in [0,2]$, there exists $C_{p}$ and $\varepsilon_0^A$ such that 
\begin{equation*}
\int_{\Td}|f-(f)_{\Td}|^{2}\le \f{1}{4C_{p}}\int_{\Td} \int_{\Td} \frac{|f(t,x) - f(t,y)|^2}{\eps^{2-\alpha}|x-y|^{\alpha}} \omega_\eps(|x-y|) \diff x \diff y     
\end{equation*}
for every $f\in L^{2}(\Td)$ and $\eps\le\eps_{0}^A$. 
\end{lemma}

For the proof, we refer to Ponce~\cite[Theorem 1.1]{MR2041005}. We also have an opposite inequality from \cite[Theorem 1]{bourgain2001another}:

\begin{lemma}\label{lem:inv_poincare_ineq} For all $\alpha \in [0,2]$, there exists a constant $C = C(\Td, \alpha)$ such that for all $f \in H^1(\Td)$
$$
\int_{\Td} \int_{\Td} \frac{|f(x)- f(y)|^2}{\eps^{2-\alpha}|x-y|^{\alpha}} \omega_{\varepsilon}(x-y) \diff x \diff y \leq C \, \|f\|_{H^1(\Td)}^2.
$$
\end{lemma}

Finally, we formulate a variant of Lemma \ref{lem:Poincare_with_average} which does not require an average on the left-hand side. The proof can be adapted from~\cite[Lemma C.3]{MR4574535}.

\begin{lemma}\label{lem:poincare_nonlocal_H1_L2}
For each $\gamma \in (0,1)$ there exists ${\varepsilon}_0^B$ and constant $C(\gamma)$ such that for all $\varepsilon \in (0, {\varepsilon}_0^B)$ and all $f \in H^1(\Td)$ we have
$$
\| f\|^2_{H^1(\Td)} \leq \gamma \int_{\Td} \int_{\Td}  \frac{|\nabla f(x) - \nabla f(y)|^2}{\eps^{2-\alpha}|x-y|^\alpha} \omega_\eps(|x-y|) \diff x \diff y + C(\gamma) \|f \|^2_{L^2(\Td)}.
$$
\end{lemma}

\bibliographystyle{abbrv}
\bibliography{fastlimit}

\end{document}